\documentclass[11pt]{article}
\pdfoutput=1

\usepackage[utf8]{inputenc}
\usepackage[a4paper,bottom=4.3cm]{geometry}
\usepackage{amssymb}
\usepackage{amsmath}
\usepackage{amsthm}
\usepackage{mathtools}
\usepackage{microtype}
\usepackage[hidelinks]{hyperref}

\newtheorem{theorem}{Theorem}
\newtheorem{lemma}[theorem]{Lemma}

\newcommand{\R}{\mathbb{R}}
\newcommand{\bOmega}{{\mathbf{\Omega}}}
\newcommand{\by}{{\mathbf{y}}}
\newcommand{\bmu}{{\boldsymbol{\mu}}}
\newcommand{\Span}{\mathop{\mathrm{span}}}
\DeclareMathOperator{\rank}{rank}
\DeclareMathOperator{\sign}{sign}
\DeclareMathOperator*{\esssup}{ess\, sup}

\title{On the approximation of vector-valued functions\\by volume sampling}
\author{Daniel Kressner\thanks{ANCHP, Institute of Mathematics, EPFL, CH-1015 Lausanne,
Switzerland}
\and Tingting Ni\thanks{SYCAMORE Lab, EPFL, CH-1015 Lausanne,
Switzerland}
\and Andr\'e Uschmajew\thanks{Institute of Mathematics \& Centre for Advanced Analytics and Predictive Sciences, University of Augsburg, 86159 Augsburg, Germany}}
\date{}

\begin{document}

\maketitle

\begin{abstract} 
Given a Hilbert space $\mathcal H$ and a finite measure space $\Omega$, the approximation of a vector-valued function $f: \Omega \to \mathcal H$ by a $k$-dimensional subspace $\mathcal U \subset \mathcal H$ plays an important role in dimension reduction techniques, such as reduced basis methods for solving parameter-dependent partial differential equations. For functions in the Lebesgue--Bochner space $L^2(\Omega;\mathcal H)$, the best possible subspace approximation error $d_k^{(2)}$ is characterized by the singular values of $f$. However, for practical reasons, $\mathcal U$ is often restricted to be spanned by point samples of $f$. We show that this restriction only has a mild impact on the attainable error; there always exist $k$ samples such that the resulting error is not larger than $\sqrt{k+1} \cdot d_k^{(2)}$. Our work extends existing results by Binev at al.~(\textit{SIAM J. Math. Anal.}, 43(3):1457--1472, 2011) on approximation in supremum norm and by Deshpande et al.~(\textit{Theory Comput.}, 2:225--247, 2006) on column subset selection for matrices.
\end{abstract} 

\section{Introduction}

Let $(\Omega,\mathcal A, \mu)$ be a finite measure space with $\mu(\Omega) > 0$. Let $\mathcal H$ be a real Hilbert space with inner product $\langle \cdot, \cdot \rangle$ and induced norm $\|\cdot \|$. In this work, we consider the approximation of a vector-valued function $f: \Omega \to \mathcal H$ by its projection to a finite-dimensional subspace $\mathcal U \subset \mathcal H$. For $1 \le p \le \infty$, possible error measures are then the approximation numbers
\begin{equation*} 
 d^{(p)}_k(f) = \inf \big\{ \| f - g \|_{L^p(\Omega;\mathcal H)}   \colon \text{$\mathcal U$ is a $k$-dimensional subspace of $\mathcal H$ and $g \in L^p(\Omega;\mathcal U)$} \big\},
\end{equation*}
where
$\|\cdot \|_{L^p(\Omega;\mathcal H)}$ defines the norm on the Lebesgue-Bochner space $L^p(\Omega;\mathcal H)$.

The number $d^{(p)}_k(f)$ measures the best possible approximation of $f$ in $L^p(\Omega;\mathcal H)$ by a function $g$ that maps (almost all points in $\Omega$) into some $k$-dimensional subspace $\mathcal U \subseteq \mathcal H$. Obviously, for given $\mathcal U$ the best choice for $g$ is $g(y) = P_{\mathcal U} f(y)$ almost everywhere, where $P_{\mathcal U}$ is the $\mathcal H$-orthogonal projection onto $\mathcal U$. With this notation, we have
\begin{equation} \label{eq:KolmogorovWidth}
d^{(p)}_k(f) = \inf \big\{ \| f- P_{\mathcal U} f \|_{L^p(\Omega;\mathcal H)}   \colon \text{$\mathcal U$ is a $k$-dimensional subspace of $\mathcal H$} \big\}.
\end{equation}
In the special case $p = \infty$, the quantity 
\[
d^{(\infty)}_k(f) = \inf \big\{ \esssup_{y \in \Omega}{} \| f(y)- P_{\mathcal U} f(y) \| \colon \text{$\mathcal U$ is a $k$-dimensional subspace of $\mathcal H$} \big\}
\]
is the Kolmogorov width~\cite{Kolmogoroff1936,Pinkus1985} of the ``essential'' image of $f$ in $\mathcal H$. For general $p < \infty$, the quantities $d^{(p)}_k(f)$ are called \emph{average} Kolmogorov widths and can even be generalized to non Hilbert-space settings, see, e.g.,~\cite[Section~2.2]{Malykhin2022} and references therein. Several results on the asymptotic order of $d^{(p)}_k(f)$ are available for the case that $\mathcal H$ is a Sobolev space and $f$ is a Gaussian random variable; see~\cite{Wang2019} for an overview.

In many situations, one hopes or even expects that the above widths decay rapidly as the dimension $k$ of the subspace increases. One then aims at constructing an actual $k$-dimensional subspace $\mathcal U$ that results in an error close to $d^{(p)}_k(f)$. A popular approach is to sample $f$ at $k$ well chosen points $y_1,\ldots, y_k \in \Omega$ and define $\mathcal U$ as the span of $f(y_1),\ldots,f(y_k)$. For example, reduced basis methods~\cite{Hesthaven2016,Quarteroni2016} for solving parameter-dependent partial differential equations commonly use a greedy strategy for selecting parameter points $y_1,\ldots, y_k$ successively. Similarly, orthogonal matching pursuit~\cite{Temlyakov2011} and the empirical interpolation method~\cite{Barrault2004} can be viewed as greedy strategies. Recently, deep neural networks have been demonstrated to be effective at approximating $f$ from (noisy) samples; see, e.g.,~\cite{Adcock2022}.

Existing convergence analyses~\cite{Barrault2004,Buffa2012,Drmac2016,Temlyakov2011} of such sample-based methods typically establish error bounds that remain \emph{qualitatively} close to~\eqref{eq:KolmogorovWidth}, but the involved prefactors can be huge, often growing exponentially with $k$. For example, the result in~\cite[Section~2]{Buffa2012} for $p=\infty$ shows that a greedy selection of $k$ points leads to an  error  (nearly) bounded by $2^{k+1}(k+1) d^{(\infty)}_k(f)$; this bound has been improved to $2^{k+1}/\sqrt{3}\cdot d^{(\infty)}_k(f)$ in \cite[Theorem 4.4]{Binev2011}. The results from~\cite{Binev2011} also show that greedy algorithms recover algebraic or exponential decays of the Kolmogorov widths.
On the other hand, the factor $2^k$ is, in general, unavoidable when using a greedy selection and this raises the question whether the use of samples necessarily leads to large prefactors or whether there exist sample selections leading to more favorable bounds~\eqref{eq:KolmogorovWidth}. Again for the case $p = \infty$ it was shown in~\cite[Theorem 4.1]{Binev2011} that there always exist sample points $y_1,\dots,y_k$ for which the prefactor can indeed be reduced to $k+1$. The proof of this result is nonconstructive and uses a maximum volume argument, a classical tool in approximation theory. A maximum volume principle also underlies a result from~\cite{Goreinov2001} on the rank-$k$ approximation of matrices in the elementwise maximum norm by cross interpolation. However, there is evidence~\cite{Deshpande2006} that such maximum volume arguments do not extend to the case $p=2$, corresponding to approximation in Frobenius norm for matrices, but randomized sampling arguments can be used instead.

\paragraph*{Main result}
In this note, we treat the case $p = 2$ and establish a prefactor $\sqrt{k+1}$.

\begin{theorem} \label{thm:main}
 Let $f \in L^2(\Omega;\mathcal H)$. Then there exists a measurable set $\hat \bOmega \subseteq \Omega \times \cdots \times \Omega = \Omega^k$ of positive product measure $\mu^{\otimes k}$ such that for $\mu^{\otimes k}$-almost every $\by = (y_1,\dots,y_k) \in \hat \bOmega$ it holds that
 \begin{equation} \label{eq:maininequality}
    d^{(2)}_k(f) \le \| f- \Pi_{\by} f \|_{L^2(\Omega;\mathcal U)} \le \sqrt{k+1} \cdot  d^{(2)}_k(f),
  \end{equation}
 where $\Pi_{\by} f$ denotes the $\mathcal H$-orthogonal projection of $f$ onto the span of $f(y_1),\ldots, f(y_k)$.
\end{theorem}

Here we do not assume $f$ to be continuous. In general, the elements $f$ of $L^2(\Omega;\mathcal H)$ are equivalence classes of functions which are only $\mu$-almost everywhere equal. The theorem is formulated in such a way that it applies to the whole equivalence class, in the sense that the set $\hat \bOmega$ is the same for every representative. Note that the set of $\by \in \hat \bOmega$ that actually satisfy~\eqref{eq:maininequality} can be different for two representatives, but only by a set of product measure zero. The rest of this note is concerned with the proof of Theorem~\ref{thm:main}.

\paragraph*{Finite-dimensional setting}
The finite-dimensional analogue of Theorem~\ref{thm:main} is due to Deshpande et al.~\cite{Deshpande2006} and reads as follows: Given a matrix $A = [a_1,\ldots, a_n] \in \R^{m\times n}$ and $k < m \le n$, there exist $k$ column indices $j_1, \ldots, j_k \in \{1,\ldots,n\}$ such that 
\begin{equation} \label{eq:ResultDeshpande}
 (\sigma_{k+1}^2 + \cdots + \sigma_m^2)^{1/2} \le \| A- P_{j_1,\ldots,j_k} A \|_{F} \le \sqrt{k+1} \cdot  (\sigma_{k+1}^2 + \cdots + \sigma_m^2)^{1/2},
 \end{equation}
where $\|\cdot\|_F$ denotes the Frobenius norm, $\sigma_1 \ge \cdots \ge \sigma_m \ge 0$ are the singular values of~$A$, and $P_{j_1,\ldots,j_k}$ denotes the orthogonal projection onto $\Span\{a_{j_1},\ldots,a_{j_k}\}$. This result is included in Theorem~\ref{thm:main} by considering the uniform measure on $\Omega = \{1,\ldots, n\}$, the Euclidean space $\mathcal H = \R^m$ and letting $f(i) = a_i$ for $i = 1,\ldots, n$. The proof in~\cite{Deshpande2006} uses a probabilistic method, showing that~\eqref{eq:ResultDeshpande} holds in expectation when columns are sampled with a probability proportional to their induced volume.

An example from~\cite{Deshpande2006} shows that the prefactor $\sqrt{k+1}$ in~\eqref{eq:ResultDeshpande} can, in general, not be improved. In follow-up work, Deshpande and Rademacher~\cite{Deshpande2010} derived a more efficient algorithm for computing a column subset selection satisfying~\eqref{eq:ResultDeshpande}; see~\cite{Cortinovis2020} for further improvements concerning its numerical realization. For fixed $j_1,\ldots,j_k$, even tiny changes of $A$ may result in a significantly larger error $\| A- P_{j_1,\ldots,j_k} A \|_{F}$~\cite{Cortinovis2023}. This lack of continuity makes it difficult to combine the algorithms from~\cite{Cortinovis2020,Deshpande2010} with a discretization of the infinite-dimensional setting. For the time being, Theorem~\ref{thm:main} is an existence result based on a probabilistic argument and it remains an open problem to design an efficient sampling procedure that leads to an error (approximately) bounded by $\sqrt{k+1} \cdot  d^{(2)}_k(f)$.

While our main strategy for the proof of Theorem~\ref{thm:main} follows~\cite{Deshpande2006}, several nontrivial modifications are needed in order to address the infinite-dimensional case.

\paragraph*{Bochner integral and measurable functions}
Before proceeding, let us recall the basic definitions regarding the Bochner integral. A function $f : \Omega \to \mathcal H$ is called \emph{strongly $\mu$-measurable} if it is the $\mu$-almost everywhere pointwise strong limit of simple functions $f_\ell : \Omega \to \mathcal H$. It is called Bochner integrable if, in addition, $\int_{\Omega} \| f - f_\ell \| \, \mathrm{d} \mu \to 0$ for $\ell \to \infty$. In this case the Bochner integral $\int_\Omega f \, \mathrm{d}\mu$ of $f$ is defined as the limit of integrals of the simple functions $f_\ell$ which then does not depend on the choice of the sequence~$f_\ell$. Since the real-valued functions $y \mapsto \| f(y) - f_\ell(y) \|$ are themselves strongly $\mu$-measurable, the integrals $\int_{\Omega} \| f - f_\ell \|$ in this definition are well-defined. In particular, the integrands are $\mathcal A$-measurable functions if $\mu$ is a complete measure, or otherwise $\mu$-almost everywhere equal to an $\mathcal A$-measurable functions~\cite[Proposition~1.1.16]{Hytonen2016}. Note that a similar reasoning will be implicitly assumed at other occurrences in the paper when composing real-valued functions from strongly $\mu$-measurable $\mathcal H$-valued functions without further mentioning.

The Hilbert space $L^2(\Omega;\mathcal H)$ consists of equivalence classes of Bochner integrable functions~$f$ for which $\int_\Omega \| f \|^2 \, \mathrm{d}\mu < \infty$. While it is common practice to not distinguish between a function and its equivalence class, we will often work with pointwise arguments of particular representatives. For more details on Bochner integrals and the spaces $L^p(\Omega;\mathcal H)$ we refer to~\cite{Hytonen2016}.

\section{Schmidt decomposition}

The Hilbert space $L^2(\Omega; \mathcal H)$ is isometrically isomorphic to the space $HS(L^2(\Omega); \mathcal H)$ of Hilbert--Schmidt operators from $L^2(\Omega)$ to $\mathcal H$. This isometry is realized by associating $f \in L^2(\Omega; \mathcal H)$ with the bounded integral operator
\[
 T_f: L^2(\Omega) \to \mathcal H, \qquad v \mapsto T_f v = \int_\Omega f v \, \mathrm{d}\mu.
\]
To see this, let $\{ u_i \colon i \in I \}$ and $\{ v_j \colon j \in J \}$ be orthonormal bases of $\mathcal H$ and $L^2(\Omega)$, respectively, with $J$ being countable. Then it can be routinely verified by properties of the Bochner integral that $\{ u_i v_j \colon i \in I, \ j \in J \}$ is an orthonormal basis of $L^2(\Omega;\mathcal H)$. The integral operator associated with $u_i v_j$ is the rank-one operator $u_i \langle v_j, \cdot \rangle_{L^2(\Omega)}$. These rank-one operators  form an orthonormal basis of $HS(L^2(\Omega);\mathcal H)$; see, e.g.,~\cite[Theorem~4.4.5]{Hsing2015}. This shows that $f \mapsto T_f$ is an isometric isomorphism between $L^2(\Omega; \mathcal H)$ and $HS(L^2(\Omega);\mathcal H)$ and leads to the following well-known decomposition of $f$.

\begin{theorem}[Schmidt decomposition]\label{thm: Schmidt decomposition}
Let $f \in L^2(\Omega; \mathcal H)$. There exist at most countable orthonormal systems $\{u_i \in \mathcal H: i=1,2,\ldots, r \}$, $\{v_i \in L^2(\Omega,\mu): i=1,2,\ldots, r \}$ (with $r \in \mathbb N \cup \{+\infty\}$) and singular values $\sigma_1\ge \sigma_2 \ge \dots$ with $\sigma_i > 0$ for $i=1,2,\dots, r$ such that 
\[
f = \sum_{i = 1}^r \sigma_i u_i v_i,
\]
with the series converging in $L^2(\Omega; \mathcal H)$. It holds that
\[
f(y) \in \overline{\Span}\{ u_i : i=1,2,\dots,r \} \quad \text{$\mu$-a.e.}
\]
and
\begin{equation}\label{eq:SVD}
f(y) = \sum_{i=1}^r \sigma_i u_i v_i(y) \quad \text{$\mu$-a.e.},
\end{equation}
with the series converging in $\mathcal H$.
\end{theorem}

\begin{proof}
By the singular value decomposition of $T = T_f$ there exist $\sigma_i$, $u_i$, and $v_i$ with the stated properties such that $T = \sum_{i = 1}^r \sigma_i u_i \langle v_i, \cdot \rangle_{L^2(\Omega)}$, where the sum converges in Hilbert--Schmidt norm; see, e.g.,~\cite[Theorem~4.3.2]{Hsing2015}. The isometric isomorphism discussed above then implies the claimed series representation of $f$ in $L^2(\Omega;\mathcal H)$.

To show the second part of the theorem, we first note that the range of $T$ is the separable Hilbert space $\mathcal H_0 = \overline{\Span}\{ u_i : i=1,2,\dots,r \}$. As a consequence,
\[
\int_A f \, \mathrm{d}\mu = \int_{\Omega} f \chi_A \, \mathrm{d}\mu = T\chi_A \in \mathcal H_0 \quad \forall A \in \mathcal A.
\]
Since the measure is finite, this implies -- by Proposition~1.2.13 in~\cite{Hytonen2016} -- that $f(y) \in \mathcal H_0$ for $\mu$-almost every $y \in \Omega$. Since $\{u_i\}$ is an orthonormal basis of $\mathcal H_0$, this in turn implies
\begin{equation}\label{eq:FourierPointwise}
f(y) = \sum_{i=1}^r \langle f(y),u_i \rangle u_i \quad \text{$\mu$-a.e.}
\end{equation}
For every $i$, the function strongly $y \mapsto \langle f(y),u_i \rangle$ is $\mu$-measurable and hence $\mu$-a.e.~equal to an $\mathcal A$-measurable function~\cite[Proposition~1.1.16]{Hytonen2016}. Then, for every $A \in \mathcal A$, 
\begin{align*}
\int_A \langle f(y), u_i \rangle \, d\mu &= \Big\langle \int_A f(y) \, d\mu, u_i \Big\rangle \\ &= \langle T \chi_A, u_i \rangle = \sum_{j=1}^r \sigma_j \langle u_j, u_i \rangle \langle v_j, \chi_A \rangle_{L^2(\Omega)}  =  \int_A \sigma_i v_i(y) \, d \mu
\end{align*}
where we have used that the Bochner integral can be interchanged with bounded linear functionals; see, e.g.,~\cite[Eq.~(1.2)]{Hytonen2016}. This implies
\[
\langle f(y), u_i \rangle = \sigma_i v_i(y) \quad \text{$\mu$-a.e.},
\]
which together with~\eqref{eq:FourierPointwise} shows~\eqref{eq:SVD}.
\end{proof}

The singular values are uniquely determined by $f$ and independent of its representative. The number $r \in \mathbb N \cup \{+\infty\}$ of positive singular values is called the rank of $f$, denoted by $\rank(f)$. By the Schmidt--Mirsky theorem~\cite[Theorem 4.4.7]{Hsing2015} and the isomorphism explained above, it follows that the truncated function
\[
f_k = \sum_{i = 1}^k \sigma_i u_i v_i = P_{{\mathcal U}_k} f,\quad {\mathcal U}_k = \Span\{u_1,\ldots, u_k\},
\]
is the best approximation of $f$ by a function of rank at most $k$ in the $L^2(\Omega;\mathcal H)$-norm. As any other projection of $f$ onto a $k$-dimensional subspace has rank at most $k$, it follows
\begin{equation} \label{eq:dk2}
 d_k^{(2)}(f) = \|f-f_k\|_{L^2(\Omega; \mathcal H)} = \big( \sigma_{k+1}^2 + \sigma_{k+2}^2  +\cdots )^{1/2}.
\end{equation}

\section{Expected volume}

For $k\ge 1$, we consider the product measure space $(\Omega^k,\mathcal A^{\otimes k}, \bmu)$, where $\mathcal A^{\otimes k}$ denotes the product $\sigma$-algebra (the smallest $\sigma$-algebra containing Cartesian products $A_1 \times \dots \times A_k$ of sets $A_1,\dots,A_k \in \mathcal A$) and $\bmu:=\mu^{\otimes k}$ denotes the product measure (the unique measure on $\mathcal A^{\otimes k}$ satisfying $\bmu(A_1 \times \dots \times A_k) = \mu(A_1) \cdots \mu(A_k)$). Let $f \in L^2(\Omega;\mathcal H)$. Given $k$ sample points $\by = (y_1,\dots,y_k)$, the Gramian of $f(y_1),\dots,f(y_k) \in \mathcal H$ is
\[
G^{(k)}(\by) =
\begin{bmatrix} \langle f(y_1), f(y_1) \rangle & \cdots & \langle f(y_1), f(y_k) \rangle \\
\vdots && \vdots \\
\langle f(y_k), f(y_1) \rangle & \cdots & \langle f(y_k), f(y_k) \rangle
\end{bmatrix}
\in \R^{k\times k}.
\]
Its determinant
\[
\det G^{(k)}(\by) = \sum_{\pi \in \mathcal S_k} \sign(\pi) \prod_{i = 1}^k \langle f(y_i), f(y_{\pi(i)}) \rangle,
\]
where $\mathcal S_k$ denotes the set of all permutations of $(1,\dots,k)$, is often called the volume of the $f(y_1),\dots,f(y_k)$. Note that in regard of the equivalence class of $f$ this quantity is only $\bmu$-almost everywhere uniquely defined. We therefore regard $\det G^{(k)}(\by)$ as an equivalence class of $\bmu$-almost everywhere equal $\bmu$-measurable real-valued functions. In turn, the Lebesgue integral of $\det G^{(k)}$, called the expected volume, can be defined. The following lemma shows that this quantity can be computed from the singular values of $f$, in analogy to~\cite[Lemma~3.1]{Deshpande2006}.

\begin{lemma} \label{lemma:integral}
With the notation introduced above, it holds that
\[
\int_{\Omega^k} \det G^{(k)} \, \mathrm{d}\bmu = \sum_{\substack{j_1,\ldots,j_k = 1 \\ \text{\rm $j_1,\dots,j_k$ mutually distinct}}}^r \sigma_{j_1}^2 \cdots \sigma_{j_k}^2.
\]
\end{lemma}

\begin{proof}
Consider the Schmidt decomposition from Theorem~\ref{thm: Schmidt decomposition}. For every finite $n \le r$, let
\[
f_n(y) = \sum_{i=1}^n \sigma_i u_i v_i(y).
\]
Since the $v_i$ are $\mathcal A$-measurable, the corresponding determinant
\[
\det G^{(k)}_n(\by) = \sum_{\pi \in \mathcal S_k} \sign(\pi) \prod_{i = 1}^k \langle f_n(y_i), f_n(y_{\pi(i)}) \rangle = \sum_{\pi \in \mathcal S_k} \sign(\pi) \prod_{i = 1}^k \sum_{j = 1}^n \sigma_j^2 v_j(y_i) v_j(y_{\pi(i)})
\]
is a nonnegative $\mathcal A^{\otimes k}$-measurable function, which satisfies, using Tonelli's theorem,
\begin{align*}
\int_{\Omega^k} \det G^{(k)}_n \, d \bmu &= \sum_{\pi \in \mathcal S_k} \sign(\pi) \int_{\Omega^k} \prod_{i = 1}^k \sum_{j = 1}^n \sigma_j^2 v_j(y_i) v_j(y_{\pi(i)}) \,\mathrm{d} \bmu \\
&= \sum_{j_1=1}^n \cdots \sum_{j_k=1}^n \sigma_{j_1}^2 \cdots \sigma_{j_k}^2  \sum_{\pi \in \mathcal S_k} \sign(\pi) \int_{\bOmega} \prod_{i=1}^k v_{j_i}(y_i) v_{j_{\pi(i)}} (y_i) \, \mathrm{d}\bmu \\
&= \sum_{j_1=1}^n \cdots \sum_{j_k=1}^n \sigma_{j_1}^2 \cdots \sigma_{j_k}^2 \sum_{\pi \in \mathcal S_k} \sign(\pi) \prod_{i=1}^k \langle v_{j_i}, v_{j_{\pi(i)}} \rangle_{L^2(\Omega)} \\
&= \sum_{j_1=1}^n \cdots \sum_{j_k=1}^n \sigma_{j_1}^2 \cdots \sigma_{j_k}^2 \sum_{\pi \in \mathcal S_k} \sign(\pi) \prod_{i=1}^k \delta_{j_i, j_{\pi(i)}}\\
&= \sum_{j_1=1}^n \cdots \sum_{j_k=1}^n \sigma_{j_1}^2 \cdots \sigma_{j_k}^2 \det (P_{j_1,\dots,j_k}),
\end{align*}
with the matrix $[P_{j_1,\dots,j_k}]_{\alpha,\beta} = \delta_{j_\alpha,j_\beta}$ for $\alpha, \beta = 1,\ldots, k$. If all $j_1,\dots,j_d$ are distinct, then $\det (P_{j_1,\dots,j_k}) = 1$, otherwise $\det (P_{j_1,\dots,j_k}) = 0$, since then $P_{j_1,\dots,j_k}$ contains identical rows. Therefore,
\[
\int_{\Omega^k} \det G^{(k)}_n \, d \bmu = \sum_{\substack{j_1,\ldots,j_k = 1 \\ \text{\rm $j_1,\dots,j_k$ mutually distinct}}}^n \sigma_{j_1}^2 \cdots \sigma_{j_k}^2.
\]
If $r$ is finite, this proves the claim by taking $n=r$. Otherwise we take the limit $n \to \infty$ and argue that
\[
\int_{\Omega^k} \det G^{(k)} \, d \bmu = \lim_{n \to \infty} \int_{\Omega^k} \det G^{(k)}_n \, d \bmu
\]
by dominated convergence. This is possible because $\det G^{(k)}_n(\by) \to \det G^{(k)}(\by)$ $\bmu$-almost everywhere (a consequence of~\eqref{eq:SVD}) and
\[
\left\lvert \det G^{(k)}_n(\by) \right\rvert \le \sum_{\pi \in \mathcal S_k} \Big\lvert \prod_{i = 1}^k \langle f_n(y_i), f_n(y_{\pi(i)}) \rangle \Big\rvert \le \sum_{\pi \in \mathcal S_k} \prod_{i=1}^k \| f_n(y_i) \|^2 \le \sum_{\pi \in \mathcal S_k}  \prod_{i=1}^k \| f(y_i) \|^2.
\]
Since $f \in L^2(\Omega;\mathcal H)$, the right hand side is a dominating integrable function on $\Omega^k$.
\end{proof}

\section{Proof of Theorem~\ref{thm:main}}

Let us assume that $\rank(f) \ge k$. Then we can define
\[
\varrho = \frac{\det G^{(k)}}{\int_{\Omega^k} \det G^{(k)} \,\mathrm{d}\bmu}.
\]
By Lemma~\ref{lemma:integral}, the denominator is positive. We have $\varrho \ge 0$ ($\bmu$-almost everywhere) and $\int_{\Omega^k} \varrho \,\mathrm{d} \bmu = 1$.  Thus, we can interpret $\varrho$ as a probability density function on the product measure space. We now aim to bound
\[
\mathbb E_\varrho\big( \|f - \Pi_\by f \|^2_{L^2(\Omega; \mathcal H)} \big) = \int_{\Omega^k} \|f - \Pi_\by f \|^2_{L^2(\Omega; \mathcal H)} \varrho(\by) \, \mathrm{d}\bmu(\by),
\]
where $\Pi_\by f: \Omega \to \mathcal H$ is the $\mathcal H$-orthogonal projection of $f$ onto $\Span\{f(y_1),\dots,f(y_k) \}$. Note that the projector $\Pi_\by$ depends on the chosen representative $f$ in the equivalence class; indeed, the subspaces spanned by $f(y_1),\dots,f(y_k)$ can be completely different for two different representatives in the same equivalence class, but only on a set of zero product measure. The statement and proof of the following theorem, which is the analog to~\cite[Theorem~1.3]{Deshpande2006}, pay attention to these subtleties.

\begin{theorem}\label{thm: expected distance}
Given $f \in L^2(\Omega;\mathcal H)$, consider the Schmidt decomposition~\eqref{eq:SVD} and assume $r = \rank(f) \ge k$. Then the following holds:
\begin{itemize}
\item[\upshape (i)]
For every fixed $\by \in \Omega^k$ the function $y' \mapsto f(y') - \Pi_\by f(y')$, with $\Pi_\by f$ defined as above, defines an equivalence class of Bochner integrable functions which belongs to $L^2(\Omega;\mathcal H)$. The nonnegative function $\by \mapsto \| f - \Pi_\by f \|_{L^2(\Omega;\mathcal H)}^2$ is therefore well-defined.
\item[\upshape (ii)]
The class of $\bmu$-almost everywhere equal functions $\by \mapsto \| f - \Pi_\by f \|_{L^2(\Omega;\mathcal H)}^2\cdot \varrho(\by)$ contains a representative which is $\mathcal A^{\otimes k}$-measurable. Hence $\mathbb E_\varrho ( \|f - \Pi_\by f \|^2_{L^2(\Omega; \mathcal H)} )$ is well-defined.
\item[\upshape (iii)]
It holds that
\begin{equation*} 
\sum_{i = k+1}^r \sigma_i^2 \le \mathbb E_\varrho\big( \|f - \Pi_\by f \|^2_{L^2(\Omega; \mathcal H)} \big) \le (k+1)\sum_{i = k+1}^r \sigma_i^2.
\end{equation*}
\end{itemize}
\end{theorem}

\begin{proof}
Ad~(i). Consider a fixed representative of $f$. Then the orthogonal projector $T = \Pi_\by$ onto $\Span\{f(y_1),\dots,f(y_k) \}$ is a bounded linear operator on $\mathcal H$. In turn, $T f$ is Bochner-integrable. 
Moreover, $\| f(y') - T f(y') \| = \|(I-T) f(y') \| \le \| f(y') \|$ for every $y' \in \Omega$. Therefore, $\int_\Omega \| f(y') - T f(y') \|^2 \, \mathrm{d}\mu \le \int_\Omega \| f(y')\|^2 \, \mathrm{d}\mu < \infty$. Obviously, for different choices of representatives $f$, the functions $f - \Pi_\by f$ are $\bmu$-almost everywhere equal.

Ad~(ii). 
Take a representative of $f$ that is everywhere a pointwise strong limit of simple functions~\cite[Proposition~1.1.16]{Hytonen2016}. 
The corresponding volume function $\by \mapsto \det G^{(k)}(\by)$ is then $\mathcal A^{\otimes k}$-measurable. Likewise, the function $(\by,y') \mapsto \det G^{(k+1)}(\by,y')$ is $\mathcal A^{\otimes (k+1)}$-measurable. By Tonelli's theorem, the nonnegative function
\[
g(\by) = \int_\Omega \det G^{(k+1)}(\by,y') \, \mathrm{d} \mu(y')
\]
is then $\mathcal A^{\otimes k}$-measurable. Since the representative $f$ is fixed, the projector $\Pi_\by$ is defined in a meaningful way for every $\by \in \Omega^k$. We claim that for every $\by$ it holds that
\begin{equation}\label{eq:g}
g(\by) = \det G^{(k)}(\by) \int_\Omega \| f(y') - \Pi_\by f(y') \|^2 \, \mathrm{d}\mu(y').
\end{equation}
After division by $\int_{\Omega^k} \det G^{(k)} \,\mathrm{d}\bmu > 0$, this implies part (ii). To establish~\eqref{eq:g}, we first note that if $f(y_1),\dots,f(y_k)$ are linearly dependent then~\eqref{eq:g} trivially holds because it implies $\det G^{(k)}(\by) = 0$ as well as $g(\by) = 0$ (since $\det G^{(k+1)}(\by,y') = 0$ for all $y' \in \Omega$). We may therefore assume that $f(y_1),\dots,f(y_k)$ are linearly independent. Then $G^{(k)}(\by)$ is invertible and for all $y' \in \Omega$ it holds that
\begin{align*}
G^{(k+1)}(\by,y') &= 
\begin{bmatrix}
G^{(k)}(\by) & w_{\by,y'} \\
w^T_{\by,y'} & \| f(y') \|^2
\end{bmatrix} \\   
&= \begin{bmatrix}
I & 0 \\
w_{\by,y'}^T G^{(k)}(\by)^{-1}   & 1
\end{bmatrix}
\begin{bmatrix}
G^{(k)}(\by) & w_{\by,y'} \\
0  & \| f(y')\|^2 - w^T_{\by,y'} G^{(k)}(\by)^{-1} w_{\by,y'}^{}
\end{bmatrix},
\end{align*}
with the vector
\[
w_{\by,y'} = [\langle f(y_1), f(y') \rangle, \ldots, \langle f(y_k), f(y') \rangle ]^T \in \R^k.
\]
It is straightforward to show that the orthogonal projection $\Pi_\by f(y')$ of $f(y')$ onto the span of $f(y_1),\dots,f(y_k)$ satisfies $\|\Pi_\by f(y')\|^2 = w^T_{\by,y'} G^{(k)}(\by)^{-1} w_{\by,y'}$. Hence we have
\[
\det G^{(k+1)}(\by,y') = \det  G^{(k)}(\by) \cdot \|f(y') - \Pi_\by f(y')\|^2,
\]
which yields~\eqref{eq:g}.

Ad~(iii). The lower bound follows immediately from~\eqref{eq:dk2}. It remains to prove the upper bound. Using the same representative for $f$ as in (ii), by~\eqref{eq:g} we have
\[
\mathbb E_\varrho\big( \|f - \Pi_\by f \|^2_{L^2(\Omega; \mathcal H)} \big) \cdot \int_{\Omega^k} \det G^{(k)} \, \mathrm{d} \bmu = \int_{\Omega^{k+1}} \det G^{(k+1)} \, \mathrm{d}\mu^{\otimes (k+1)}.
\]
Applying Lemma~\ref{lemma:integral} to both sides gives
\begin{equation}\label{eq:ExpectedValue}
\mathbb E_\varrho\big( \|f - \Pi_\by f \|^2_{L^2(\Omega; \mathcal H)} \big) = \frac{\sum_{j_1,\ldots,j_{k+1}} \sigma_{j_1}^2 \cdots \sigma_{j_k}^2 \sigma_{j_{k+1}}^2}{\sum_{j_1,\ldots,j_{k}} \sigma_{j_1}^2 \cdots \sigma_{j_k}^2},
\end{equation}
where the summations are performed over mutually distinct indices ranging from $1$ to~$r$. As there are $(k+1)!$ different ways of arranging $k+1$ mutually distinct numbers, one deduces that
\begin{align*}
\sum_{j_1,\ldots,j_{k+1}} \sigma_{j_1}^2 \cdots \sigma_{j_k}^2 \sigma_{j_{k+1}}^2 & = (k+1)! \sum_{j_1 < \cdots < j_{k+1}} \sigma_{j_1}^2 \cdots \sigma_{j_k}^2 \sigma_{j_{k+1}}^2 \\
&= (k+1)!  \sum_{j = k+1}^r \sigma^2_j \sum_{j_1 < \cdots < j_{k} < j} \sigma_{j_1}^2 \cdots \sigma_{j_k}^2 \\
&\le (k+1)!  \sum_{j = k+1}^r \sigma^2_j \sum_{j_1 < \cdots < j_{k}} \sigma_{j_1}^2 \cdots \sigma_{j_k}^2 \\
&= (k+1) \sum_{j = k+1}^r \sigma^2_j \sum_{j_1,\ldots,j_{k}} \sigma_{j_1}^2 \cdots\sigma_{j_k}^2.
\end{align*}
Inserting this inequality into~\eqref{eq:ExpectedValue} completes the proof.
\end{proof}

We are ready to state the main result of this section, which in light of~\eqref{eq:dk2} also proves Theorem~\ref{thm:main}.

\begin{theorem}
Let $f \in L^2(\Omega;\mathcal H)$ with rank $r \in \mathbb N \cup \{\infty\}$. There exists a measurable set $\hat \bOmega \subseteq \Omega^k$ of positive product measure such that for every representative $f$ in the equivalence class and $\bmu$-almost every sample tuple $\by \in \hat \bOmega$ it holds that
\[
\|f - \Pi_\by f \|^2_{L^2(\Omega; \mathcal H)} \le (k+1)\sum_{i = k+1}^r \sigma_i^2.
\]
\end{theorem}

\begin{proof}
For the moment, let us assume that $r  \ge k$. In this case, the proof boils down to the fact that the probability for a random variable to not being larger than its expected value is positive. For completeness, we provide the full argument. As argued in the proof of Theorem~\ref{thm: expected distance}(ii) there exists a suitable representative $f$ such that the two functions $\by \mapsto \varrho(\by)$ and $\by \mapsto \|f - \Pi_\by f \|^2_{L^2(\Omega; \mathcal H)} \cdot \varrho(\by)$ are both $\mathcal A^{\otimes k}$-measurable. Hence, the sets $B = \{ \by \in \Omega^k \colon \varrho(\by) > 0 \}$ and
\begin{align*}
\hat \bOmega &= \left\{ \by \in B \colon \|f - \Pi_\by f \|^2_{L^2(\Omega; \mathcal H)} \le (k+1)\sum_{i = k+1}^r \sigma_i^2 \right\} \\
&= \left\{ \by \in B \colon  \|f - \Pi_\by f \|^2_{L^2(\Omega; \mathcal H)} \cdot \varrho(\by) - (k+1)\sum_{i = k+1}^r \sigma_i^2 \cdot \varrho(\by) \le 0 \right\}
\end{align*}
both belong to $\mathcal A^{\otimes k}$. It is then a standard argument to show that we must have $\bmu(\hat \bOmega) > 0$, since otherwise Theorem~\ref{thm: expected distance}(iii) would be violated. For any other representative of $f$, the values of $\| f - \Pi_\by f \|^2_{L^2(\Omega;\mathcal H)}$ remain the same for $\bmu$-almost all $\by$. This completes the proof in the case $r \ge k$.

In the case $\rank(f) = r < k$ we can apply the theorem to $r$ instead of $k$. Then
\[
\| f - \Pi_{(y_1,\dots,y_r)} f \|_{L^2(\Omega;\mathcal H)}^2 \le (r+1)\sum_{i = r+1}^r \sigma_i^2 = 0
\]
for all tuples $(y_1,\dots,y_r)$ in a subset $\tilde \bOmega$ of $\Omega^r$ with positive $\mu^{\otimes r}$-measure. Then for all $\by \in \tilde \bOmega \times \Omega^{k-r}$, which has positive $\bmu$-measure, we also have
\[
\| f - \Pi_\by f \|_{L^2(\Omega;\mathcal H)}^2 \le \| f - \Pi_{(y_1,\dots,y_r)} f \|_{L^2(\Omega;\mathcal H)}^2 = 0
\]
since the subspaces on which one projects only get larger.
\end{proof}

\paragraph*{Acknowledgements} The work of A.U.~was supported by the Deutsche Forschungsgemeinschaft (DFG, German Research Foundation) – Projektnummer 506561557. The authors thank the reviewers for their detailed reading and constructive feedback.

\small
\bibliography{main}

\begin{thebibliography}{10}

\bibitem{Adcock2022}
B.~Adcock, S.~Brugiapaglia, N.~Dexter, and S.~Morage.
\newblock Deep neural networks are effective at learning high-dimensional
  {H}ilbert-valued functions from limited data.
\newblock In J.~Bruna, J.~Hesthaven, and L.~Zdeborova, editors, {\em
  Proceedings of the 2nd Mathematical and Scientific Machine Learning
  Conference}, volume 145 of {\em Proceedings of Machine Learning Research},
  pages 1--36. PMLR, 2022.

\bibitem{Barrault2004}
M.~Barrault, Y.~Maday, N.~C. Nguyen, and A.~T. Patera.
\newblock An `empirical interpolation' method: application to efficient
  reduced-basis discretization of partial differential equations.
\newblock {\em C. R. Math. Acad. Sci. Paris}, 339(9):667--672, 2004.

\bibitem{Binev2011}
P.~Binev, A.~Cohen, W.~Dahmen, R.~DeVore, G.~Petrova, and P.~Wojtaszczyk.
\newblock Convergence rates for greedy algorithms in reduced basis methods.
\newblock {\em SIAM J. Math. Anal.}, 43(3):1457--1472, 2011.

\bibitem{Buffa2012}
A.~Buffa, Y.~Maday, A.~T. Patera, C.~Prud'homme, and G.~Turinici.
\newblock {\it {A} priori} convergence of the greedy algorithm for the
  parametrized reduced basis method.
\newblock {\em ESAIM Math. Model. Numer. Anal.}, 46(3):595--603, 2012.

\bibitem{Cortinovis2020}
A.~Cortinovis and D.~Kressner.
\newblock Low-rank approximation in the {F}robenius norm by column and row
  subset selection.
\newblock {\em SIAM J. Matrix Anal. Appl.}, 41(4):1651--1673, 2020.

\bibitem{Cortinovis2023}
Alice Cortinovis.
\newblock Personal communication, 2023.

\bibitem{Deshpande2010}
A.~Deshpande and L.~Rademacher.
\newblock Efficient volume sampling for row/column subset selection.
\newblock In {\em 2010 {IEEE} 51st {A}nnual {S}ymposium on {F}oundations of
  {C}omputer {S}cience---{FOCS} 2010}, pages 329--338. 2010.

\bibitem{Deshpande2006}
A.~Deshpande, L.~Rademacher, S.~Vempala, and G.~Wang.
\newblock Matrix approximation and projective clustering via volume sampling.
\newblock {\em Theory Comput.}, 2:225--247, 2006.

\bibitem{Drmac2016}
Z.~Drma\v{c} and S.~Gugercin.
\newblock A new selection operator for the discrete empirical interpolation
  method---improved a priori error bound and extensions.
\newblock {\em SIAM J. Sci. Comput.}, 38(2):A631--A648, 2016.

\bibitem{Goreinov2001}
S.~A. Goreinov and E.~E. Tyrtyshnikov.
\newblock The maximal-volume concept in approximation by low-rank matrices.
\newblock In {\em Structured matrices in mathematics, computer science, and
  engineering, {I} ({B}oulder, {CO}, 1999)}, volume 280 of {\em Contemp.
  Math.}, pages 47--51. Amer. Math. Soc., Providence, RI, 2001.

\bibitem{Hesthaven2016}
J.~S. Hesthaven, G.~Rozza, and B.~Stamm.
\newblock {\em Certified reduced basis methods for parametrized partial
  differential equations}.
\newblock Springer, Cham, 2016.

\bibitem{Hsing2015}
T.~Hsing and R.~Eubank.
\newblock {\em Theoretical foundations of functional data analysis, with an
  introduction to linear operators}.
\newblock John Wiley \& Sons, Ltd., Chichester, 2015.

\bibitem{Hytonen2016}
T.~Hyt\"{o}nen, J.~van Neerven, M.~Veraar, and L.~Weis.
\newblock {\em Analysis in {B}anach spaces. {V}ol.~{I}. {M}artingales and
  {L}ittlewood-{P}aley theory}.
\newblock Springer, Cham, 2016.

\bibitem{Kolmogoroff1936}
A.~Kolmogoroff.
\newblock \"{U}ber die beste {A}nn\"{a}herung von {F}unktionen einer gegebenen
  {F}unktionenklasse.
\newblock {\em Ann. of Math. (2)}, 37(1):107--110, 1936.

\bibitem{Malykhin2022}
Y.~Malykhin.
\newblock Widths and rigidity, 2022.
\newblock arXiv:2205.03453.

\bibitem{Pinkus1985}
A.~Pinkus.
\newblock {\em {$n$}-widths in approximation theory}.
\newblock Springer-Verlag, Berlin, 1985.

\bibitem{Quarteroni2016}
A.~Quarteroni, A.~Manzoni, and F.~Negri.
\newblock {\em Reduced basis methods for partial differential equations}.
\newblock Springer, Cham, 2016.

\bibitem{Temlyakov2011}
V.~Temlyakov.
\newblock {\em Greedy approximation}.
\newblock Cambridge University Press, Cambridge, 2011.

\bibitem{Wang2019}
H.~Wang.
\newblock Probabilistic and average linear widths of weighted {S}obolev spaces
  on the ball equipped with a {G}aussian measure.
\newblock {\em J. Approx. Theory}, 241:11--32, 2019.

\end{thebibliography}

\end{document}